\documentclass[12pt,reqno]{amsart}
\usepackage{amsmath, amssymb, latexsym, amscd, amsthm,amsfonts,amstext}
\usepackage[mathscr]{eucal}
\usepackage{xspace}
\usepackage{array, tabularx, longtable}
\usepackage{multicol,color}
\usepackage{indentfirst}
\usepackage{graphics}
\usepackage[colorlinks=true]{hyperref}
\theoremstyle{plain}
\newtheorem{theorem}{Theorem}
\newtheorem{corollary}[theorem]{Corollary}

\newtheorem{lemma}[theorem]{Lemma}
\newtheorem{proposition}[theorem]{Proposition}
\theoremstyle{remark}

\newtheorem{claim}{Claim}

\theoremstyle{definition}

\begin{document}
\date{} 

\title [Gauss map of complete minimal surfaces on annular ends ]{Ramification of the Gauss map \\of complete minimal surfaces in $\mathbb R^3$ and $\mathbb R^4$ \\on annular ends}

\author{Gerd Dethloff and Pham Hoang Ha}

\keywords{Minimal surface, Gauss map, Ramification, Value distribution theory.}
\subjclass[2010]{Primary 53A10; Secondary 53C42, 30D35, 32H30}

\begin{abstract} 
In this article, we study the ramification of the Gauss map of complete minimal surfaces in $\mathbb R^3$ and $ \mathbb R^4$ on annular ends. We obtain results which are similar to the ones obtained by Fujimoto (\cite{Fu1}, \cite{Fu2}) and Ru (\cite{Ru1}, \cite{Ru2}) for (the whole) complete minimal surfaces, thus we show that the restriction of the Gauss map to an annular end of such a complete minimal surface cannot have more branching (and in particular not avoid more values) than on the whole complete minimal surface. We thus give an improvement of the results on annular ends of complete minimal surfaces of 
Kao~(\cite{Kao}). \\

\noindent RESUME. Dans ce travail nous obtenons des th\'eor\`{e}mes de ramification de 
l'application de Gauss de certaines classes de surfaces minimales compl\`{e}tes dans $\mathbb R^3$ et $ \mathbb R^4$.
\end{abstract}
\maketitle
\tableofcontents
\section{Introduction}
Let $M$ be a non-flat minimal surface in $\mathbb R^3,$ or more precisely, a non-flat connected oriented minimal surface in $\mathbb R^3.$ By definition, the Gauss map $G$ of $M$ is a map which maps each point $p \in M$ to the unit normal vector $G(p) \in S^2$ of $M$ at $p.$ Instead of $G,$ we study the map $g:= \pi\circ G: M \rightarrow  \mathbb P^1(\mathbb C))$, 
where  $\pi : S^2 \rightarrow \mathbb P^1(\mathbb C)$
is the stereo\-graphic projection.  By associating a holomorphic local coordinate $z=u+\sqrt{-1}v$ with each positive isothermal coordinate system $(u, v),$ $M$ is considered as an open Riemann surface with a conformal metric $ds^2$ and by the assumption of minimality of $M,$ $g$ is a meromorphic function on $M.$ 

In 1988, H. Fujimoto (\cite{Fu1}) proved Nirenberg's conjecture that if $M$ is a complete non-flat minimal surface in $\mathbb R^3,$ then its Gauss map can omit at most 4 points, and the bound is sharp. After that, he also extended that result for  minimal surfaces in $\mathbb R^m.$ \\
\indent In 1993, M. Ru (\cite{Ru2}) refined these results  by studying the Gauss maps of minimal surfaces in $\mathbb R^m$ with ramification. But for our purpose, we here only introduce the case $m =3.$ To give that result, we recall some definitions.\\
\indent One says that $g: M \rightarrow \mathbb P^1(\mathbb C)$ is {\it  ramified over a point} $a=(a_0: a_{1}) \in  \mathbb P^{1}(\mathbb C)$ {\it with multiplicity at least} $e$ if all the zeros of the function $ a_0g_1 - a_{1}g_{0}$ have orders at least $e,$ where $g =(g_0: g_{1})$ is a reduced representation.  If the image of $g$ omits $a,$ one will say that $g$ is {\it ramified over a with multiplicity }$\infty.$
Ru proved : \\
 {\bf Theorem A. }\ {\it Let M be a non-flat complete minimal surface in $\mathbb R^3.$ If there are $q\  (q > 4)$
distinct points $a^1, ... , a^q \in \mathbb P^1(\mathbb C)$ such that the Gauss map of M is ramified over
$a^j$ with multiplicity at least $m_j$ for each j, then $ \sum_{j=1}^q(1 - \frac{1}{m_j})\leq 4.$ }

To prove this result, he constructed a pseudo-metric with negative curvature with ramification and used the previous argument of Fujimoto.

On the other hand, in 1991, S. J. Kao (\cite{Kao}) used the ideas of Fujimoto~(\cite{Fu1}) to show that the Gauss map of an end of a non-flat complete minimal surface in $\mathbb R^3$ that is conformally an annulus $\{ z  : 0 < 1/r < |z| < r \}$ must also assume every value, with at most 4 exceptions. In 2007, L. Jin and M. Ru (\cite{JR}) extended Kao's result to minimal surfaces in $\mathbb R^m.$ Kao (\cite{Kao}) proved : \\
{\bf Theorem B. }\ {\it The Gauss map $g$ on an annular end 
of a non-flat complete minimal surface in $\mathbb R^3$
assumes every value on the unit sphere infinitely often, with the possible exception of four values.}\\

A natural question is whether a result as in Theorem~A for the rami\-fication of the Gauss map still holds on an annular end of  a non-flat complete minimal surface in $\mathbb R^3$. In this paper we give an affirmative answer :
\begin{theorem} \label{T1}
Let $M$ be a non-flat complete minimal surface in $\mathbb R^3$ and let $A \subset M$ be an
annular end of $M$ which is conformal to $\{z  :  0 < 1/r < |z| < r\}$, where $z$
is a conformal coordinate. If there are $q\  (q > 4)$
distinct points $a^1, ... , a^q \in \mathbb P^1(\mathbb C)$ such that the restriction of the Gauss map of $M$ to $A$ is ramified over
$a^j$ with multiplicity at least $m_j$ for each $j$,   then 
\begin{equation} \label{1} \sum_{j=1}^q(1 - \frac{1}{m_j})\leq 4.
\end{equation}  
Moreover, (\ref{1})  still holds if we replace, for all $j=1,...,q$,  $m_j$ by the limit inferior
of the orders of the zeros of the function $ a_0^jg_1 - a_{1}^jg_{0}$ on $A$ (where $g =(g_0: g_{1})$ is a reduced representation and $a^j=(a^j_0:a^j_1)$) and in particular 
by $\infty$ if $g$ takes the value $a^j$ only a finite number of times on $A$.
\end{theorem}
Theorem \ref{T1} gives in particular the following generalization of Theorem B of Kao :
\begin{corollary}\label{C2}
{\it If the Gauss map $g$ on an annular end 
of a non-flat complete minimal surface in $\mathbb R^3$
assumes four values on the unit sphere only finitely often, it takes any other value infinitely often without ramification.}
\end{corollary}

Moreover, we also would like to consider  the Gauss map of complete minimal surfaces $M$ immersed in $\mathbb R^4,$ this case has been investigated by various authors (see, for example Osserman (\cite{Os}), Chen (\cite{Chen}), Fujimoto (\cite{Fu2}) and Kawakami (\cite{K})). In this case, the Gauss map of $M$ may be identified with a pair of meromorphic functions $g=(g^1, g^2)$ which is introduced in $\S 4$ (see also Osserman (\cite{Os}), Fujimoto (\cite{Fu2}) and Kawakami (\cite{K})). We shall prove the following result
which again shows that the restriction of the Gauss map to an annular end cannot have
more branching (and in particular not avoid more values) than on the whole complete minimal surface immersed in $\mathbb R^4$ : 
\begin{theorem}\label{T2}
Suppose that $M$ is a complete non-flat minimal surface in $\mathbb R^4$ and $g=(g^1, g^2)$ is the Gauss map of $M.$ Let A be an annular end of $M$ which is conformal to $\{z  :  0 < 1/r < |z| < r\}$, where $z$
is a conformal coordinate. Let $a^{11}, ... , a^{1q_1}$ respectively $a^{21}, ... , a^{2q_2}$ be $ q_1\ (q_1>2)$ respectively  $q_2\ (q_2>2)$ 
distinct points in $ \mathbb P^1(\mathbb C).$ \\
\indent (i) In the case $g^l\not\equiv constant\  (l=1,2),$ if $g^l$ is ramified over $a^{lj}$ with multiplicity at least $m_{lj}$ for each $j\ (l=1,2)$ on $A,$  then\\
$ \gamma_1=\sum_{j=1}^{q_1}(1 - \frac{1}{m_{1j}})\leq 2,$ or $ \gamma_2=\sum_{j=1}^{q_2}(1 - \frac{1}{m_{2j}})\leq 2,$ or\\
$$\dfrac{1}{\gamma_1 - 2} + \dfrac{1}{\gamma_2 - 2} \geq 1.$$
\indent (ii) In the case where  $g^1$ or $g^2$ is constant, say $g^2\equiv constant, $ if $g^1$  is ramified over $a^{1j}$ with multiplicity at least $m_{1j}$ for each $j,$ we have the following :
$$ \gamma_1=\sum_{j=1}^{q_1}(1 - \frac{1}{m_{1j}})\leq 3.$$
Moreover, the results  still hold if we replace, for all $a^{lj}$ $(j=1,...,q_l; \:l=1,2)$  the $m_{lj}$ by the limit inferior
of the orders of the zeros of the function $ a^{lj}_0g^l_1 - a^{lj}_{1}g^l_{0}$ on $A$ (where $g^l =(g^l_0: g^l_{1})$ are reduced representations and $a^{lj}=(a^{lj}_0:a^{lj}_1)$) and in particular 
by $\infty$ if $g^l$ takes the value $a^{lj}$ only a finite number of times on $A$.
\end{theorem}
The main idea to prove the theorems is to  construct a pseudo-metric with negative curvature with ramification on an annular end, which is a refinement of  the ideas in Ru (\cite{Ru2}). After that we use  arguments similar to those used by Kao (\cite{Kao}) and by Fujimoto (\cite{Fu1}, \cite {Fu3}) to finish the proofs. 

\section{Auxiliary  lemmas}
Let $f$ be a nonconstant holomorphic map of a disk $\Delta_R := \{z \in \mathbb C: |z| < R \}$ into $\mathbb P^1(\mathbb C),$ where $0 < R < \infty.$ Take a reduced representation $f = (f_0: f_1)$ on $\Delta_R$ and define
$$||f||:= (|f_0|^2 + |f_1|^2)^{1/2}, W(f_0,f_1) =W_z(f_0, f_1):= f_0f_1' - f_1f_0',$$
where the derivatives are taken with respect to the variable $z$.
Let $a^j\: ( 1\leq j \leq q )$ be $q$ distinct points in $\mathbb P^1(\mathbb C).$ We may assume $a^j=(a^j_0: a^j_1 )$ with $|a^j_0|^2 + |a^j_1|^2 = 1$ $( 1\leq j \leq q ),$ and set
$$F_j:=a^j_0f_1 - a^j_1f_0 \ (1\leq j \leq q).$$
\begin{proposition}(\cite[Proposition 2.1.6 and 2.1.7]{Fu3})\label{W}.\\
a) If $\xi$ is another local coordinate, then 
$W_{\xi}(f_0,f_1)= W_z(f_0,f_1) \cdot (\frac{dz}{d\xi})$.\\
b)  $W(f_0,f_1) \not\equiv 0$ (iff $f$is nonconstant).
\end{proposition}
\begin{proposition}(\cite[Proposition 2.1]{Fu2}).
\label{P1}
For each $\epsilon > 0$ there exist positive constants $C_1$ and $\mu$ depending only on $a^1, \cdots, a^q$ and on $\epsilon$ respectively such that
$$\Delta \log \bigg( \dfrac{||f||^\epsilon}{\Pi_{j=1}^q\log(\mu||f||^{2}/|F_j|^2)}\bigg)\geq \dfrac{C_1||f||^{2q-4}|W(f_0,f_1)|^2}{\Pi_{j=1}^q|F_j|^2\log^2(\mu||f||^{2}/|F_j|^2)}$$
\end{proposition}
\begin{lemma}\label{L2}
Suppose that $q-2 -\sum_{j=1}^q\frac{1}{m_j} > 0$ and that $f$ is ramified over $a^j$ with multiplicity at least $m_j$ for each $j$ $ (1\leq j \leq q).$  Then there exist positive constants $C$ and $\mu ( > 1)$ depending only on $a^j$ and $m_j$ $ (1\leq j \leq q)$ which satisfy the following : If we set 
$$v := \dfrac{C||f||^{q-2 -\sum_{j=1}^q\frac{1}{m_j}}|W(f_0,f_1)|}{\Pi_{j=1}^q|F_j|^{1-\frac{1}{m_j}}\log(\mu||f||^{2}/|F_j|^2)}$$
on $\Delta_R \setminus \{F_1\cdot ...\cdot F_q = 0\}$ and $v=0$ on $\Delta_R \cap \{F_1 \cdot ... \cdot F_q = 0\},$ then $v$ is continuous on $\Delta_R$ and satisfies the condition $$\Delta\log v \geq v^2$$ in the sense of distribution.
\end{lemma}
\begin{proof}
First, we prove the continuousness of $v.$\\
Obviously, $v$ is continuous on $\Delta_R \setminus \{F_1\cdot ...\cdot F_q = 0\}.$\\
Take a point $z_0$ with $F_i(z_0) = 0$ for some $i.$ Then $F_j(z_0) \not= 0$ for all $j\not= i$ and $\nu_{F_i}(z_0) \geq m_i.$ Changing indices if necessary, we may assume that $f_0(z_0) \not= 0,$ then $a^i_0 \not= 0.$ Hence, we get 
$$\nu_W(z_0) = \nu_{ \dfrac{( a^i_0\frac{f_1}{f_0}-a^i_1 )'}{ a^i_0}}(z_0) =\nu_{ \dfrac{(F_i/f_0)'}{ a^i_0}}(z_0) = \nu_{F_i}(z_0) - 1.$$
Thus, 
\begin{align*}
\nu_{v\Pi_{j=1}^q\log(\mu||f||^{2}/|F_j|^2)}(z_0) &= \nu_W(z_0) - \sum_{j=1}^q(1-\frac{1}{m_j})\nu_{F_j}(z_0)\\  
& = \nu_{F_i}(z_0) - 1 - (1 - \frac{1}{m_i})\nu_{F_i}(z_0)\\
& = \dfrac{\nu_{F_i}(z_0)}{m_i} - 1 \geq 0. \ (*)
\end{align*}
So, $\lim_{z\rightarrow z_0}v(z) = 0.$ This implies that $v$ is continuous on $\Delta_R.$\\
Now, we choose constants $C$ and $\mu$ such that $C^2$ and $\mu$ satisfy the inequality in Proposition \ref{P1} for the case $\epsilon = q-2 -\sum_{j=1}^q\frac{1}{m_j}.$ Then we have (by using  that $|F_j| \leq ||f||$ $ ( 1\leq j \leq q )$) :
\begin{align*}
\Delta\log v &\geq \Delta\log\dfrac{||f||^{q-2 -\sum_{j=1}^q\frac{1}{m_j}}}{\Pi_{j=1}^q\log(\mu||f||^{2}/|F_j|^2)} \\ 
& \geq C^2\dfrac{||f||^{2q-4}|W(f_0,f_1)|^2}{\Pi_{j=1}^q|F_j|^2\log^2(\mu||f||^{2}/|F_j|^2)}\\
& \geq C^2\dfrac{||f||^{2q-4-2\sum_{j=1}^q\frac{1}{m_j}}|W(f_0,f_1)|^2}{\Pi_{j=1}^q|F_j|^{2- \frac{2}{m_j}}\log^2(\mu||f||^{2}/|F_j|^2)}\\
&=v^2 .
\end{align*}
Thus Lemma \ref{L2} is proved.
\end{proof}
\begin{lemma}(Generalized Schwarz's Lemma (\cite{Ah})). \label{L3}
Let $v$ be a nonnegative real-valued continuous subharmonic function on $\Delta_R.$ If $v$ satisfies the inequality  $\Delta\log v \geq v^2$ in the sense of distribution, then
$$v(z) \leq \dfrac{2R}{R^2 - |z|^2}\, .$$
\end{lemma}
\begin{lemma} \label{L4}
 For every $\delta$ with $q-2 -\sum_{j=1}^q\frac{1}{m_j} >  q\delta > 0 $ and $f$ which is ramified over $a^j$ with multiplicity at least $m_j$ for each $j$ $ (1\leq j \leq q),$ there exists a positive constant $C_0$ such that
$$\dfrac{||f||^{q-2 -\sum_{j=1}^q\frac{1}{m_j} -  q\delta}|W(f_0,f_1)|}{\Pi_{j=1}^q|F_j|^{1-\frac{1}{m_j}-\delta}}\leq C_0\dfrac{2R}{R^2 - |z|^2}\, .$$
\end{lemma}
\begin{proof}
By using an argument as in (*) of the proof of Lemma \ref{L2}, the above inequality is correct on $\{F_1\cdot ... \cdot F_q = 0\}$ for every $C_0 > 0$ (the left hand side of the above inequality is zero).\\
If $z \not\in \{F_1\cdot ...\cdot F_q = 0\}, $ using Lemma \ref{L2} and Lemma \ref{L3}, we get
$$\dfrac{C||f||^{q-2 -\sum_{j=1}^q\frac{1}{m_j}}|W(f_0,f_1)|}{\Pi_{j=1}^q|F_j|^{1-\frac{1}{m_j}}\log(\mu||f||^{2}/|F_j|^2)} \leq \dfrac{2R}{R^2 - |z|^2}\, ,$$
where $C$ and $\mu$ are the constants given in Lemma \ref{L2}.\\
On the other hand, for a given $\delta > 0, $ it holds that
$$\lim_{x \rightarrow 0}x^{\delta}\log (\mu/x^2) < + \infty \, ,  $$
so we can set 
$$\overline{C}:= \sup_{0<x\leq 1}x^{\delta}\log (\mu/x^2) ( < + \infty )\, . $$
Then we have  
\begin{align*}
&\dfrac{||f||^{q-2 -\sum_{j=1}^q\frac{1}{m_j} -  q\delta}|W(f_0,f_1)|}{\Pi_{j=1}^q|F_j|^{1-\frac{1}{m_j}-\delta}}\\ 
&= \dfrac{||f||^{q-2 -\sum_{j=1}^q\frac{1}{m_j}}|W(f_0,f_1)|}{\Pi_{j=1}^q|F_j|^{1-\frac{1}{m_j}}}\prod_{j=1}^q\bigg(\dfrac{|F_j|}{||f||}\bigg)^{\delta}\\
&=\dfrac{||f||^{q-2 -\sum_{j=1}^q\frac{1}{m_j}}|W(f_0,f_1)|}{\Pi_{j=1}^q|F_j|^{1-\frac{1}{m_j}}\log(\mu||f||^{2}/|F_j|^2)}\prod_{j=1}^q(\dfrac{|F_j|}{||f||})^{\delta}\log(\mu||f||^{2}/|F_j|^2)\\ 
&\leq \dfrac{\overline{C}^q||f||^{q-2 -\sum_{j=1}^q\frac{1}{m_j}}|W(f_0,f_1)|}{\Pi_{j=1}^q|F_j|^{1-\frac{1}{m_j}}\log(\mu||f||^{2}/|F_j|^2)}\\ 
&\leq \dfrac{\overline{C}^q}{C}\dfrac{2R}{R^2 - |z|^2}\, .
\end{align*}
This proves Lemma \ref{L4}.
\end{proof}
We finally will need the following result on completeness of open Riemann surfaces with conformally flat metrics due to Fujimoto :
\begin{lemma}(\cite[Lemma 1.6.7]{Fu3}). \label{L5}
Let $d\sigma^2$ be a conformal flat metric on an open Riemann surface $M$. Then for every point $p \in M$, there is a holomorphic and locally biholomorphic map $\Phi$ of
a disk (possibly with radius $\infty$)  $\Delta_{R_0} := \{w : |w|<R_0 \}$ $(0<R_0 \leq \infty )$ onto an open neighborhood of $p$ with $\Phi (0) = p$ such that $\Phi$ is a local isometry, namely the pull-back 
$\Phi^*(d\sigma^2)$ is equal to the standard (flat) metric on $\Delta_{R_0}$, and for some point $a_0$ with $|a_0|=1$, the $\Phi$-image of the curve 
$$L_{a_0} : w:= a_0 \cdot s \; (0 \leq s < R_0)$$
is divergent in $M$ (i.e. for any compact set $K \subset M$, there exists an $s_0<R_0$
such that the $\Phi$-image of the curve $L_{a_0} : w:= a_0 \cdot s \; (s_0 \leq s < R_0)$
does not intersect $K$).
\end{lemma}

\section{The proof of  Theorem \ref{T1}}
\begin{proof}
\indent For convenience of the reader, we first recall some notations on the Gauss map of minimal surfaces in $\mathbb R^3.$
\indent Let $x =( x_1, x_2, x_3) : M \rightarrow \mathbb R^3$ be a non-flat complete minimal surface and $g: M \rightarrow  \mathbb P^1(\mathbb C)$ its Gauss map.  Let $z$ be a local holomorphic coordinate. Set $\phi_i := \partial x_i / \partial z \ (i = 1, 2, 3 )$ and $\phi:= \phi_1-\sqrt{-1}\phi_2.$ Then, the (classical) Gauss map $g: M \rightarrow \mathbb P^1(\mathbb C)$ is given by  $$g=\dfrac{\phi_3}{\phi_1 - \sqrt{-1}\phi_2},$$
and the metric on $M$ induced from $\mathbb R^3$ is given by
$$ds^2= |\phi|^2(1 + |g|^2)^2|dz|^2  \text{ (see Fujimoto (\cite{Fu3}))}.$$
We remark that although the $\phi_i$, $(i=1,2,3)$ and $\phi$ depend on $z$, $g$ and $ds^2$ do not.
Next we take a reduced representation $g = (g_0 : g_1)$ on $M$ and set $||g|| = (|g_0|^2 +|g_1|^2)^{1/2}.$ Then we can rewrite 
\begin{equation}
ds^2 = |h|^2||g||^4|dz|^2\,, 
\end{equation}
where $h:= \phi/g_0^2.$ In particular, $h$ is a holomorphic map without zeros. We remark that  $h$ depends on $z$, however, the reduced representation $g=(g_0:g_1)$ is globally defined on $M$ and independent of $z$.
Finally we observe that by the assumption that $M$ is not flat,  $g$ is not constant.

Now the proof of Theorem \ref{T1} will be given in four steps :

{\bf Step 1:}  We will fix notations on the annular end $A \subset M$. Moreover, by passing  to  a sub-annular end of $A \subset M$ we simplify the geometry of the theorem.

Let $A \subset M$ be an annular end of $M,$ that is, $A = \{z  :  0 < 1/r < |z| < r < \infty  \},$ where $z$ is a (global) conformal coordinate of $A$. Since $M$ is complete with respect to $ds^2$, we may assume that the restriction of
 $ds^2$ to $A$ is complete on the set  $\{ z : |z| = r\}$, i.e., the set $\{ z : |z| = r\}$ is at infinite distance from any  point  of $A$. 
 
Let $a^j\:( 1\leq j \leq q )$ be $q >4$ distinct points in $\mathbb P^1(\mathbb C).$ We may assume $a^j=(a^j_0: a^j_1 )$ with $|a^j_0|^2 + |a^j_1|^2 = 1$ $( 1\leq j \leq q ),$ and we set
$G_j:=a^j_0g_1 - a^j_1g_0 \ (1\leq j \leq q)$ for the reduced representation $g = (g_0 : g_1)$ of the Gauss map. By the identity theorem, the $G_j$ have at most countable many
zeros. Let $m_j$ be the  limit inferior
of the orders of the zeros of the functions $G_j$ on $A$ (and in particular 
 $m_j = \infty$ if $G_j$ has only a finite number of zeros on $A$). 
 
 All the $m_j$ are increasing if we only consider the zeros which the functions $G_j$ take on a subset $B \subset A$. So without loss of generality we may prove our theorem only
 on a sub-annular end, i.e. a subset $ A_t :=\{z  :  0 < t \leq |z| < r < \infty  \} \subset A$ 
 with some $t$ such that $1/r < t<r$. (We trivially observe that for $c:=tr>1$, $s:= r/\sqrt{c}$, $\xi := z/ \sqrt{c}$, we have $A_t = \{\xi  :  0 < 1/s \leq |\xi| < s < \infty  \}$.) 
 
 By passing to such a sub-annular end we will be able to extend the construction of a metric in step 2 below to the set $\{ z : |z|=1/r \}$, and, moreover, we may assume that for all $j=1,...,q$ :
 $$
 g\: {\rm omits}\: a^j\: (m_j=\infty)\: {\rm or} \:{\rm takes}\: a^j \:{\rm infinitely}\: {\rm often}
 \:{\rm with}\:{\rm ramification } 
$$
 \begin{equation} \label{ass1}
 2 \leq m_j< \infty \:{\rm and}\: {\rm is}\: 
 {\rm ramified}\: {\rm over}\: a^j\: 
 {\rm with}\: {\rm multiplicity}\: {\rm at}\: {\rm least}\: m_j.
 \end{equation}
 
 {\bf Step 2:} On the annular end $A = \{z  :  0 < 1/r \leq |z| < r < \infty  \}$ minus a discrete
 subset $S \subset A$ we construct a flat metric $d\tau^2$ on $A \setminus S$  which
 is complete on the set  $\{ z : |z| = r\} \cup S$, i.e., the set $\{ z : |z| = r\} \cup S $ is at infinite distance from any  point  of $A\setminus S$.
 
 We may assume that 
 \begin{equation} \label{ass2} \sum_{j=1}^q(1 - \frac{1}{m_j}) > 4,
\end{equation}  
since otherwise Theorem \ref{T1} is already proved.

Take $\delta$ with
$$ \dfrac{q -4-\sum_{j=1}^q\frac{1}{m_j}}{q} > \delta > \dfrac{q -4-\sum_{j=1}^q\frac{1}{m_j}}{q +2}, $$
and set $p = 2/ (q -2-\sum_{j=1}^q\frac{1}{m_j}-q\delta).$ Then 
\begin{equation} \label{3.4.1}0 < p < 1 , \ \frac{p}{1-p} > \frac{\delta p}{1-p} > 1 \ . 
\end{equation}
 Consider the  subset
$$ A_1 = A \setminus \{ z : W_z(g_0, g_1)(z) = 0 \}$$
of $A$. We define a new metric
$$d\tau^2 = |h|^{\frac{2}{1-p}}\bigg(\dfrac{\Pi_{j=1}^q|G_j|^{1-\frac{1}{m_j}-\delta}}{|W(g_0,g_1)|}\bigg)^{\frac{2p}{1-p}}|dz|^2 \ $$
on $A_1$ (where again $G_j := a^j_0g_1 - a^j_1g_0$ and $h$ is defined with respect
to the coordinate $z$ on $A_1 \subset A$ and $W(g_0,g_1) = W_z(g_0,g_1)$) : \\
First we observe that $d\tau$ is continuous and nowhere vanishing on $A_1.$ Indeed, $h$ is without zeros on $A$ and for each $z_0 \in A_1$ with $G_j(z_0) \not= 0$ for all $j=1,...,q$,  $d\tau$ is continuous at $z_0.$\\
Now, suppose there exists a point $z_0\in A_1$ with $G_j(z_0) = 0$ for some $j.$ Then $G_i(z_0) \not= 0$ for all $i\not= j$ and $\nu_{G_j}(z_0) \geq m_j \geq 2$. Changing the indices if necessary, we may assume that $g_0(z_0) \not= 0$, so also $a^j_0 \not= 0.$ So, we get 
\begin{equation} \label{order}
\nu_{W(g_0,g_1)}(z_0) = \nu_{ \dfrac{( a^j_0\frac{g_1}{g_0}-a^j_1 )'}{ a^j_0}}(z_0) =\nu_{ \dfrac{(G_j/g_0)'}{ a^j_0}}(z_0) = \nu_{G_j}(z_0) - 1 > 0.
\end{equation}
This is in contradition with $z_0 \in A_1$. Thus, $d\tau$ is continuous and nowhere vanishing on $A_1.$ \\
Next, it is easy to see that $d\tau$ is flat.\\
By Proposition \ref{W} a) and the dependence of $h$ on $z$ and the independence of the $G_j$ of $z$, we also easily see that $d\tau$ is independent of the choice of the coordinate $z$.

The key point is to prove the following claim :
\begin{claim} \label{Cl1}
 $d\tau$ is complete on the set  $\{ z : |z| = r\}\cup \{z : W(g_0, g_1)(z) = 0\},$ i.e., the set $\{ z : |z| = r\}\cup \{z : W(g_0, g_1)(z) = 0 \}$ is at infinite distance from any interior point in $A_1.$ 
\end{claim}
 If $W(g_0, g_1)(z_0) = 0,$ then we have two cases :\\
{\it Case 1.} $G_j(z_0) = 0$ for some $j \in \{1, 2, ..., q \}.$\\
Then we have $G_i(z_0) \not= 0$ for all $i\not= j$ and $\nu_{G_j}(z_0) \geq m_j.$ By the same argument as in (\ref{order}) we get that
$$\nu_{W(g_0,g_1)}(z_0) =  \nu_{G_j}(z_0) - 1.$$
Thus (since $m_j \geq 2$), 
\begin{align*}
\nu_{d\tau}(z_0) &= \dfrac{p}{1-p}((1-\frac{1}{m_j}-\delta)\nu_{G_j}(z_0)-\nu_{W(g_0,g_1)}(z_0))\\  
& = \dfrac{p}{1-p}(1-(\frac{1}{m_j}+\delta)\nu_{G_j}(z_0))
\leq  \dfrac{p}{1-p}(1-(\frac{1}{m_j}+\delta)m_j)\\
&\leq -\dfrac{2\delta p}{1-p}.
\end{align*}
{\it Case 2.} $G_j(z_0) \not= 0$ for all $ 1\leq j \leq q.$ \\
It is easily to see that  $\nu_{d\tau}(z_0) \leq -\dfrac{p}{1-p}.$\\
So, since $0<\delta<1$, we can find a positive constant $C$ such that 
$$|d\tau| \geq  \dfrac{C}{ |z-z_0|^{\delta p/(1-p)}}|dz| $$ 
in a neighborhood of $z_0$. Combining with (\ref{3.4.1}) we thus have that $d\tau$ is complete on $\{z : W(g_0,g_1)(z) = 0\}.$ \\
\indent Now assume that $d\tau$ is not complete on $\{z: |z| = r \}.$ Then there exists $\gamma: [0, 1) \rightarrow A_1,$ where $\gamma (1) \in \{z : |z| = r \},$ so that $|\gamma| < \infty.$ Furthermore, we may also assume that $dist(\gamma (0); \{z : |z| = 1/r\}) > 2|\gamma|.$ Consider a small disk $\Delta$ with center at $\gamma (0).$ Since $d\tau$ is flat, $\Delta$ is isometric to an ordinary disk in the plane (cf. e.g.  Lemma \ref{L5}). Let $\Phi: \{|w| < \eta \}\rightarrow \Delta$ be this isometry. Extend $\Phi,$ as a local isometry into $A_1,$ to the largest disk $\{|w| < R\} = \Delta_R$ possible. Then $R \leq |\gamma|.$ The reason that $\Phi$ cannot be extended to a larger disk is that the image goes to the outside boundary $\{z : |z| = r\}$ of $A_1$ (it cannot go to points of $A$ with $W(g_0,g_1)=0$ since we have shown already the completeness of $A_1$ with respect 
to these points). More precisely, there exists a point $w_0$ with $|w_0| =R$ so that $\Phi(\overline{0,w_0}) = \Gamma_0$ is a divergent curve on $A.$\\
The map $\Phi(w)$ is locally biholomorphic, and the metric on $\Delta_R$ induced from $ds^2$ through $\Phi$ is given by
\begin{equation} \label{3.3} \Phi^*ds^2 =  |h \circ \Phi|^2||g \circ \Phi||^4|\frac{dz}{dw}|^2|dw|^2 \ .
\end{equation}
On the other hand, $\Phi$ is isometric, so we have
$$ |dw| = |d\tau|= \bigg(\dfrac{|h| \Pi_{j=1}^q|G_j|^{(1-\frac{1}{m_j}-\delta)p}}{|W(g_0,g_1)|^p}\bigg)^{\frac{1}{1-p}}|dz|$$
$$\Rightarrow |\dfrac{dw}{dz}|^{1-p} = \dfrac{|h| \Pi_{j=1}^q|G_j|^{(1-\frac{1}{m_j}-\delta)p}}{|W(g_0,g_1)|^p}.$$
Set $f:= g(\Phi), f_0 := g_0(\Phi), f_1 := g_1(\Phi)$ and $F_j:= G_j(\Phi).$  Since
$$W_w(f_0, f_1) = (W_z(g_0, g_1) \circ \Phi)\frac{dz}{dw}, $$
we obtain
\begin{equation} \label{3.4} |\dfrac{dz}{dw}| = \dfrac{|W(f_0,f_1)|^p}{|h(\Phi)| \Pi_{j=1}^q|F_j|^{(1-\frac{1}{m_j}-\delta)p}}
\end{equation}
By (\ref{3.3}) and (\ref{3.4}) and by definition of $p$, therefore, we get
\begin{align*}
\Phi^*ds^2& = \bigg( \dfrac{||f||^2|W(f_0,f_1)|^p}{ \Pi_{j=1}^q|F_j|^{(1-\frac{1}{m_j}-\delta)p}}\bigg)^2|dw|^2\\
&= \bigg( \dfrac{||f||^{q-2-\sum_{j=1}^q\frac{1}{m_j}-q\delta}|W(f_0,f_1)|}{ \Pi_{j=1}^q|F_j|^{1-\frac{1}{m_j}-\delta}}\bigg)^{2p}|dw|^2.
\end{align*}
Using the Lemma \ref{L4}, we obtain
$$\Phi^*ds^2 \leqslant C^{2p}_0.(\dfrac{2R}{R^2 -|w|^2})^{2p}|dw|^2.$$
 Since $0 < p < 1,$ it then follows that 
$$d_{\Gamma_0} \leqslant \int_{\Gamma_0}ds = \int_{\overline{0,w_0}}\Phi^*ds \leqslant C^p_0. \int_0^R(\dfrac{2R}{R^2 -|w|^2})^{p}|dw|  < + \infty, $$
where $d_{\Gamma_0}$ denotes the length of the divergent curve $\Gamma_0$ in $M,$ contradicting the assumption of completeness of $M.$ Claim \ref{Cl1} is proved.\\
To summarize, in step 2 we have constructed, for $A = \{z  :  0 < 1/r \leq |z| < r < \infty  \}$
and $S = \{ z : W_z(g_0, g_1)(z) = 0 \}$, a continuous and nowhere vanishing metric 
$d\tau^2$ on $A \setminus S$ which is flat, independent of the choice of coordinate $z$,
and complete with respect to the points of $S$ and with respect to the (outside) boundary
$\{ z : |z| = r\}$.

{\bf Step 3:} We will "symmetrize" the metric constructed in step 2 so that it will become
a complete and flat metric on $Int(A) \setminus (S \cup \tilde{S})$ (with $\tilde{S}$ another discrete subset).

We introduce a new coordinate $\xi (z) :=1/z$ .
By Proposition \ref{W}~a) we have $S= \{ z : W_z(g_0, g_1)(z) = 0 \} = 
\{ z : W_{\xi}(g_0, g_1)(z) = 0 \} $ (where the zeros are taken with the same multiplicities) and since $d\tau^2$ is independent of the coordinate $z$, the change of coordinate $\xi (z) = 1/z$ yields an isometry of $A \setminus S$
onto the set $\tilde{A} \setminus \tilde{S}$, where $\tilde{A}:=\{z : 1/r < |z| \leq r \}$ 
and $\tilde{S}:= \{ z : W_z(g_0, g_1)(1/z) = 0 \}$.
In particular we have (if still $\tilde{h}$ is defined with respect to the coordinate $\xi$) :
$$d\tau^2 = |\tilde{h}(1/z)|^{\frac{2}{1-p}}\bigg(\dfrac{\Pi_{j=1}^q|G_j(1/z)|^{1-\frac{1}{m_j}-\delta}}{|W_{(1/z)}(g_0,g_1)(1/z)|}\bigg)^{\frac{2p}{1-p}}|d(1/z)|^2 $$
 $$= \bigg( |h(1/z)|^{\frac{2}{1-p}}\bigg(\dfrac{\Pi_{j=1}^q|G_j(1/z)|^{1-\frac{1}{m_j}-\delta}}{|W_z(g_0,g_1)(1/z)|}\bigg)^{\frac{2p}{1-p}} |\frac{dz}{d(1/z)}|^2 \bigg) |d(1/z)|^2$$
 $$= |h(1/z)|^{\frac{2}{1-p}}\bigg(\dfrac{\Pi_{j=1}^q|G_j(1/z)|^{1-\frac{1}{m_j}-\delta}}{|W_z(g_0,g_1)(1/z)|}\bigg)^{\frac{2p}{1-p}} |dz|^2$$

\indent We now define 
\begin{align*}
d\tilde{\tau}^2&= \bigg(|h(z)h(1/z)| \cdot \dfrac{\Pi_{j=1}^q|G_j(z)G_j(1/z)|^{(1-\frac{1}{m_j}-\delta)p}}{|W_z(g_0,g_1)(z) \cdot W_z(g_0,g_1)(1/z)|^p}\bigg)^{\frac{2}{1-p}}|dz|^2\\
&=\lambda^2(z)|dz|^2, 
\end{align*}
on $\tilde{A}_1 := \{z : 1/r < |z| < r \} \setminus \{ z : W_z(g_0, g_1)(z)\cdot W_z(g_0, g_1)(1/z) = 0 \} $. Then $d\tilde{\tau}^2$ is complete  on $\tilde{A}_1$ : In fact by what we showed
above we have:  Towards any point of the boundary 
$\partial \tilde{A}_1 := \{z : 1/r = |z| \} \cup \{z :  |z| = r \} \cup \{ z : W_z(g_0, g_1)(z)\cdot W_z(g_0, g_1)(1/z) = 0 \} $ of $\tilde{A}_1$, one of the factors of $\lambda^2(z)$ is bounded
from below away from zero, and 
the
other factor is the one of a complete metric with respect of this part of the boundary.
Moreover by  the corresponding properties of the two factors of $\lambda^2(z)$ it is trivial that $d\tilde{\tau}^2$ is a continuous nowhere vanishing and flat metric  on $\tilde{A}_1$.

{\bf Step 4 :} We produce a contradiction by using Lemma \ref{L5} to the open Riemann surface $(\tilde{A}_1, d\tilde{\tau}^2)$ : \\
In fact, we apply Lemma \ref{L5} to any point $p \in \tilde{A}_1$. Since $d\tilde{\tau}^2$ is 
complete, there cannot exist a divergent curve from $p$ to the boundary $\partial \tilde{A}_1$
with finite length with respect to $d\tilde{\tau}^2$. Since $\Phi : \Delta_{R_0} \rightarrow \tilde{A}_1$ is a local
isometry, we necessarily have $R_0 = \infty$. So $\Phi : {\mathbb C} \rightarrow \tilde{A}_1 \subset \{z : |z| <r\}$ is a non constant holomorphic map, which contradicts to
Liouville's theorem. So our assumption (\ref{ass2}) was wrong.
This proves the Theorem~\ref{T1}. 
\end{proof}

\section{The proof of  Theorem \ref{T2} }
\begin{proof}
 For convenience of the reader, we first recall some notations on the Gauss map of minimal surfaces in $\mathbb R^4$.
Let $x=(x_1, x_2, x_3, x_4) : M \rightarrow \mathbb R^4$ be a non-flat complete minimal surface in $\mathbb R^4.$ As  is well-known, the set of all oriented 2-planes in $\mathbb R^4$ is canonically identified with the quadric  
$$Q_2(\mathbb C):= \{(w_1:...: w_4) | w^2_1 +  ... + w^2_4 = 0\}$$
in $\mathbb P^3(\mathbb C).$ By definition, the Gauss map $g : M \rightarrow Q_2(\mathbb C)$ is the map which maps each point $p$ of $M$ to the point of $Q_2(\mathbb C)$ corresponding to the oriented tangent
plane of $M$ at $p.$ The quadric $Q_2(\mathbb C)$ is biholomorphic to $\mathbb P^1(\mathbb C)\times\mathbb P^1(\mathbb C) .$ By suitable identifications we may regard $g$ as a pair of meromorphic functions $g=(g^1, g^2)$ on $M.$  Let $z$ be a local holomorphic coordinate. 
Set $\phi_i := \partial x_i/dz$ for $i=1,...,4.$ Then, $g^1$ and $g^2$ are given by 
$$g^1 = \dfrac{\phi_3 + \sqrt{-1}\phi_4}{\phi_1 - \sqrt{-1}\phi_2},\  g^2 = \dfrac{-\phi_3 + \sqrt{-1}\phi_4}{\phi_1 - \sqrt{-1}\phi_2}$$
and the metric on $M$ induced from $\mathbb R^4$ is given by
$$ds^2 =|\phi|^2(1 +|g^1|^2)(1+|g^2|^2)|dz|^2 ,$$
where $\phi:= \phi_1 - \sqrt{-1}\phi_2.$
We remark that although the $\phi_i$, $(i=1,2,3,4)$ and $\phi$ depend on $z$, $g=(g^1,g^2)$ and $ds^2$ do not.
Next we take reduced representations $g ^l= (g^l_0 : g^l_1)$ on $M$ and set $||g^l|| = (|g^l_0|^2 +|g^l_1|^2)^{1/2}$ for $l=1,2.$ Then we can rewrite 
\begin{equation}ds^2 = |h|^2||g^1||^2||g^2||^2|dz|^2 \,,
\end{equation}
 where $h:= \phi/(g^1_0g^2_0)$.
 In particular, $h$ is a holomorphic map without zeros. We remark that  $h$ depends on $z$, however, the reduced representations $g^l=(g^l_0:g^l_1)$ are globally defined on $M$ and independent of $z$.
Finally we observe that by the assumption that $M$ is not flat,  $g$ is not constant.

Now the proof of Theorem \ref{T2} will be given in four steps :

{\bf Step 1:} This step is completely analogue to step 1 in the proof of Theorem \ref{T1}.
We get : By passing to a sub-annular end we may assume that the annular end is 
 $A = \{z  :  0 < 1/r \leq  |z| < r < \infty  \},$ where $z$ is a (global) conformal coordinate of $A$, that the restriction of
 $ds^2$ to $A$ is complete on the set  $\{ z : |z| = r\}$, i.e., the set $\{ z : |z| = r\}$ is at infinite distance from any  point  of $A$, 
  and, moreover, that for all $j=1,...,q_l$, $l=1,2$ (case (i)) respectively for all $j=1,...,q_1$,
  $l=1$
  (case (ii)), we have :
$$
 g^l\: {\rm omits}\: a^{lj}\: (m_{lj}=\infty)\: {\rm or} \:{\rm takes}\: a^{lj} \:{\rm infinitely}\: {\rm often}
 \:{\rm with}\:{\rm ramification}\: \:
$$
 \begin{equation} \label{ass1'}
2 \leq m_{lj}< \infty
 \:{\rm and}\: {\rm is}\: 
 {\rm ramified}\: {\rm over}\: a^{lj}\: 
 {\rm with}\: {\rm multiplicity}\: {\rm at}\: {\rm least}\: m_{lj}.
  \end{equation}
 
 From now on we separate the two cases (i) and (ii), dealing first with the case (i).
  
 {\bf Step 2 for the case (i):} Our strategy is the same as for step 2 in the proof of
 Theorem \ref{T1}. 
We may assume that  $ \gamma_1=\sum_{j=1}^{q_1}(1 - \frac{1}{m_{1j}})> 2$,  $\gamma_2 =\sum_{j=1}^{q_2}(1 - \frac{1}{m_{2j}})> 2,$ and
\begin{equation}\label{ass4}
\dfrac{1}{\gamma_1 - 2} + \dfrac{1}{\gamma_2 - 2} < 1\,,
\end{equation}
since otherwise case (i) of Theorem \ref{T2} is already proved.

Choose  $\delta_0 (> 0)$ such that $\gamma_l - 2 - q_l\delta_0 > 0$ for all $l=1,2,$ and
$$\dfrac{1}{\gamma_1 - 2 - q_1\delta_0} + \dfrac{1}{\gamma_2 - 2- q_2\delta_0} = 1.$$
If we choose a positive constant $\delta (< \delta_0)$ sufficiently near to $\delta_0$ and set $$p_l := 1/(\gamma_l - 2 - q_l\delta),(l=1, 2),$$ we have
\begin{equation} \label{4.2}
0 < p_1 + p_2 < 1, \ \dfrac{\delta p_l}{1-p_1-p_2} > 1 \:( l=1,2)  \ .
\end{equation}
 Consider the  subset
$$ A_2 = A \setminus \{ z : W_z(g^1_0, g^1_1)(z) \cdot W_z(g^2_0, g^2_1)(z) = 0 \}$$
of $A$.  We define a new metric
$$d\tau^2 =  \bigg(|h|\dfrac{\Pi_{j=1}^{q_1}|G^1_j|^{(1-\frac{1}{m_{1j}}-\delta)p_1}\Pi_{j=1}^{q_2}|G^2_j|^{(1-\frac{1}{m_{2j}}-\delta)p_2}}{|W(g^1_0,g^1_1)|^{p_1}|W(g^2_0,g^2_1)|^{p_2}}\bigg)^{\frac{2}{1-p_1 -p_2}}|dz|^2 \ $$
on $A_2$ 
(where again $G^l_j := a^{lj}_0g^l_1 - a^{lj}_1g^l_0\: ( l= 1,2)$ and $h$ is defined with respect
to the coordinate $z$ on $A_2 \subset A$ and $W(g^l_0,g^l_1) = W_z(g^l_0,g^l_1)$).

It is easy to see that by the same arguments as in step 2 of the proof of Theorem \ref{T1}
(applied for each $l=1,2$), we get that $d\tau$ is a continuous nowhere vanishing and flat metric on $A_2$, which is
moreover independant of the choice of the coordinate $z$.

The key point is to prove the following claim :
\begin{claim} \label{Cl2}
 $d\tau^2$ is complete on the set $\{ z : |z| = r\}\cup \{z : \Pi_{l=1,2}W(g^l_0, g^l_1)(z) = 0 \},$ i.e., the set $\{ z : |z| = r\}\cup \{z : \Pi_{l=1,2}W(g^l_0, g^l_1)(z)=0 \}$ is at infinite distance from any interior point in $A_2$. 
\end{claim}
It is easy to see that by the same method as in the proof of Claim~\ref{Cl1} in the proof of Theorem \ref{T1}, we may show that $d\tau$ is complete on $\{z : \Pi_{l=1,2}W(g^l_0, g^l_1)(z) = 0 \}.$

Now assume $d\tau$ is not complete on $\{z : |z| = r \}.$ Then there exists $\gamma: [0, 1) \rightarrow A_2,$ where $\gamma (1) \in \{z : |z| = r \}$, so that $|\gamma| < \infty.$ Furthermore, we may also assume that $dist(\gamma (0), \{z : |z| = 1/r\}) > 2|\gamma|.$ Consider a small disk $\Delta$ with center at $\gamma (0).$ Since $d\tau$ is flat, $\Delta$ is isometric to an ordinary disk in the plane. Let $\Phi: \{|w| < \eta \}\rightarrow \Delta$ be this isometry. Extend $\Phi$, as a local isometry into $A_2,$ to the largest disk $\{|w| < R\} = \Delta_R$ possible. Then $R \leq |\gamma|.$ The reason that $\Phi$ cannot be extended to a larger disk is that the image goes to the outside boundary $\{z : |z| = r\}$ of $A_2.$ More precisely, there exists a point $w_0$ with $|w_0| =R$ so that $\Phi(\overline{0,w_0}) = \Gamma_0$ is a divergent curve on $A.$\\
The map $\Phi(w)$ is locally biholomorphic, and the metric on $\Delta_R$ induced from $ds^2$ through $\Phi$ is given by
\begin{equation} \label{4.3} \Phi^*ds^2 =  |h \circ \Phi|^2||g^1 \circ \Phi||^2||g^2 \circ \Phi||^2|\frac{dz}{dw}|^2|dw|^2 \  .
\end{equation}
On the other hand, $\Phi$ is isometric, so we have
$$ |dw| = |d\tau|= \bigg(|h|\dfrac{\Pi_{j=1}^{q_1}|G^1_j|^{(1-\frac{1}{m_{1j}}-\delta)p_1}\Pi_{j=1}^{q_2}|G^2_j|^{(1-\frac{1}{m_{2j}}-\delta)p_2}}{|W(g^1_0,g^1_1)|^{p_1}|W(g^2_0,g^2_1)|^{p_2}}\bigg)^{\frac{1}{1-p_1 -p_2}}|dz|$$
$$\Rightarrow |\dfrac{dw}{dz}|^{1-p_1-p_2} = |h|\dfrac{\Pi_{j=1}^{q_1}|G^1_j|^{(1-\frac{1}{m_{1j}}-\delta)p_1}\Pi_{j=1}^{q_2}|G^2_j|^{(1-\frac{1}{m_{2j}}-\delta)p_2}}{|W(g^1_0,g^1_1)|^{p_1}|W(g^2_0,g^2_1)|^{p_2}}.$$
For each $l=1,2,$ we set $f^l:= g^l(\Phi), f^l_0 := g^l_0(\Phi), f^l_1 := g^l_1(\Phi)$ and $F^l_j:= G^l_j(\Phi).$  Since
$$W_w(f^l_0, f^l_1) = (W_z(g^l_0, g^l_1) \circ \Phi)\frac{dz}{dw}, (l=1,2), $$
we obtain
\begin{equation} \label{4.4} 
 |\dfrac{dz}{dw}| = \dfrac{\Pi_{l=1,2}|W(f^l_0,f^l_1)|^{p_l}}{|h(\Phi)|\Pi_{l=1,2}\Pi_{j=1}^{q_l}|F^l_j|^{(1-\frac{1}{m_{lj}}-\delta)p_l}}\ .
 \end{equation}
By (\ref{4.3}) and (\ref{4.4}), we get
\begin{align*}
\Phi^*ds^2& = \bigg(\Pi_{l=1,2} \dfrac{||f^l||(|W(f^l_0,f^l_1)|)^{p_l}}{ \Pi_{j=1}^{q_l}|F^l_j|^{(1-\frac{1}{m_{lj}}-\delta)p_l}}\bigg)^2|dw|^2\\
&= \Pi_{l=1,2}\bigg( \dfrac{||f^l||^{q_l-2-\sum_{j=1}^{q_l}\frac{1}{m_{lj}}-q_l\delta}|W(f^l_0,f^l_1)|}{ \Pi_{j=1}^q|F^l_j|^{1-\frac{1}{m_{lj}}-\delta}}\bigg)^{2p_l}|dw|^2.
\end{align*}
Using the Lemma \ref{L4}, we obtain
$$\Phi^*ds^2 \leqslant C^{2(p_1+p_2)}_0.(\dfrac{2R}{R^2 -|w|^2})^{2(p_1+p_2)}|dw|^2.$$
 Since $0 < p_1+p_2 < 1$ by (\ref{4.2}), it then follows that 
$$d_{\Gamma_0} \leqslant \int_{\Gamma_0}ds = \int_{\overline{0,w_0}}\Phi^*ds \leqslant C^{p_1+p_2}_0. \int_0^R(\dfrac{2R}{R^2 -|w|^2})^{p_1+p_2}|dw|  < + \infty, $$
where $d_{\Gamma_0}$ denotes the length of the divergent curve $\Gamma_0$ in $M,$ contradicting the assumption of completeness of $M.$ Claim \ref{Cl2} is proved.

 {\bf Steps 3 and 4 for the case (i):}
 These steps are  analogue to the corresponding steps in the proof of Theorem \ref{T1}.
Define $d\tilde{\tau}^2 =\lambda^2(z) |dz|^2 \ $ on 
$$\tilde{A}_2 := \{ z : 1/r < |z| < r \}  \setminus $$
$$\setminus \{ z : W_z(g^1_0, g^1_1)(z) \cdot W_z(g^2_0, g^2_1)(z) \cdot W_z(g^1_0, g^1_1)(1/z) \cdot W_z(g^2_0, g^2_1)(1/z)= 0 \} \, , 
$$
 where
\begin{align*}
\lambda(z)& = \bigg(|h(z)|\dfrac{\Pi_{j=1}^{q_1}|G^1_j(z)|^{(1-\frac{1}{m_{1j}}-\delta)p_1}\Pi_{j=1}^{q_2}|G^2_j(z)|^{(1-\frac{1}{m_{2j}}-\delta)p_2}}{|W_z(g^1_0,g^1_1)(z)|^{p_1}|W_z(g^2_0,g^2_1)(z)|^{p_2}}\bigg)^{\frac{1}{1-p_1 -p_2}}\\
& \times \bigg(|h(1/z)|\dfrac{\Pi_{j=1}^{q_1}|G^1_j(1/z)|^{(1-\frac{1}{m_{1j}}-\delta)p_1}\Pi_{j=1}^{q_2}|G^2_j(1/z)|^{(1-\frac{1}{m_{2j}}-\delta)p_2}}{|W_z(g^1_0,g^1_1)(1/z)|^{p_1}|W_z(g^2_0,g^2_1)(1/z)|^{p_2}}\bigg)^{\frac{1}{1-p_1 -p_2}}.
\end{align*}
By using Claim \ref{Cl2}, the continuous nowhere vanishing and flat metric $d\title{\tau}$
on $A_2$ is also complete.  Using the identical argument of step 4 in the proof of Theorem \ref{T1}
to the open Riemann surface $(\tilde{A}_2, d\tilde{\tau})$ produces a contradiction, so assumption (\ref{ass4}) was wrong.
This implies case (i) of the Theorem \ref{T2}.\\

\indent We finally consider the case (ii) of Theorem \ref{T2} (where $g^2 \equiv constant$ and $g^1 \not\equiv constant$). Suppose that 
$\gamma_1 > 3.$ We can choose $\delta$ with 
$$ \dfrac{\gamma_1 - 3}{q_1} > \delta > \dfrac{\gamma_1 - 3}{q_1 +1}, $$
and set $p = 1/ (\gamma_1 - 2 -q_1\delta).$ Then 
$$0 < p < 1 , \ \frac{p}{1-p} > \frac{\delta p}{1-p} > 1 . $$
Set $$d\tau^2 = |h|^{\frac{2}{1-p}}\bigg(\dfrac{\Pi_{j=1}^{q_1}|G^1_j|^{1-\frac{1}{m_{1j}}-\delta}}{|W(g^1_0,g^1_1)|}\bigg)^{\frac{2p}{1-p}}|dz|^2 .\ $$
Using this metric, by the analogue arguments as in step 2 to step 4 of the proof of Theorem \ref{T1},  we get the case (ii) of Theorem \ref{T2}.
\end{proof}

{\bf Acknowledgements.} 
This work was completed during a stay of the authors at 
the Vietnam Institute for Advanced Study in Mathematics (VIASM).
The research of the second named author is partially supported by a NAFOSTED grant of Vietnam.

\vspace{1cm}

\noindent {\it Gerd Dethloff $^{1,2}$and Pham Hoang Ha$^{3}$\\

 \noindent $^1$ Universit\'e Europ\'eenne de Bretagne, France\\
 $^2$ Universit\'e de Brest \\
 Laboratoire de Math\'{e}matiques de Bretagne Atlantique - \\
 UMR CNRS 6205\\ 
6, avenue Le Gorgeu, BP 452\\
29275 Brest Cedex, France\\
$^3$ Department of Mathematics\\
Hanoi National University of Education\\
136 XuanThuy str., Hanoi, Vietnam}\\

\noindent Email : Gerd.Dethloff@univ-brest.fr ; phamhoangha23@gmail.com
\end{document}